\documentclass[a4paper, 12pt]{amsart}

\usepackage{amssymb, amsthm}
\usepackage[backref]{hyperref}
\usepackage[alphabetic,backrefs,lite]{amsrefs}
\usepackage{amscd}   
\usepackage[all]{xy} 
\usepackage{graphicx}
\usepackage{mathrsfs}
\usepackage{enumerate}
\usepackage{url}

\usepackage{tikz}
\usepackage{verbatim}
\usepackage{multirow,qtree}

  \tikzstyle{block} = [rectangle, draw,
      text width=10em, text centered, minimum height=4em]
  \tikzstyle{line} = [draw, -latex']

\usepackage[alphabetic,backrefs,lite]{amsrefs}
\usepackage{amscd}   
\usepackage[all]{xy} 
\usepackage{graphicx}
\usepackage{mathrsfs}
\definecolor{labelkey}{rgb}{0,0,1}
\usepackage{}

\usepackage{tikz}

\newtheorem{lemma}{Lemma}[section]
\newtheorem{theorem}[lemma]{Theorem}
\newtheorem{proposition}[lemma]{Proposition}
\newtheorem{prop}[lemma]{Proposition}
\newtheorem*{prop*}{Proposition}
\newtheorem{cor}[lemma]{Corollary}
\newtheorem{conj}[lemma]{Conjecture}

\newtheorem{claim*}{Claim}

\newenvironment{proof3}{\emph{Proof of Theorem 3.1:}}{\hfill$\square$}

\theoremstyle{definition}
\newtheorem{remark}[lemma]{Remark}

\newcommand{\F}{{\mathbb F}}

\newcommand{\Q}{{\mathbb Q}}

\newcommand{\Z}{{\mathbb Z}}

\newcommand{\calN}{{\mathcal N}}
\newcommand{\calO}{{\mathcal O}}

\newcommand{\frakf}{{\mathfrak f}}

\newcommand{\frakp}{{\mathfrak p}}
\newcommand{\frakq}{{\mathfrak q}}

\newcommand{\frakN}{{\mathfrak N}}

\newcommand{\frakP}{{\mathfrak P}}

\DeclareMathOperator{\nr}{Norm}
\DeclareMathOperator{\tr}{Tr}

\DeclareMathOperator{\Frob}{Frob}

\DeclareMathOperator{\Aut}{Aut}

\DeclareMathOperator{\Res}{Res}
\DeclareMathOperator{\Norm}{ Norm}

\DeclareMathOperator{\SL}{SL}
\DeclareMathOperator{\GL}{GL}

\DeclareMathOperator{\Tors}{Tors}
\usepackage{color}

\numberwithin{equation}{section}
\numberwithin{table}{section}

\title{On Ternary Diophantine Equations of Signature $(p,p,3)$ over Number Fields}

\author{Erman Isik}
\author{ Yasemin Kara}
\author{ Ekin Ozman} 
  
    \address{Bogazici University, Istanbul, Turkey}
  
  \email{erman.isik@boun.edu.tr}
  
     \address{Bogazici University, Istanbul, Turkey}
  \email{yasemin.kara@boun.edu.tr}
  \address{Bogazici University, Istanbul, Turkey\\
  University of Texas at Austin-USA}
  
  \email{ekin.ozman@boun.edu.tr}


\begin{document}

         \maketitle
         \begin{abstract} 
         In this paper, we prove results about solutions of the Diophantine equation $x^p+y^p=z^3$ over various number fields using the modular method.  Firstly, by assuming some standard modularity conjecture we prove an asymptotic result for general number fields of narrow class number one satisfying some technical conditions.  Secondly, we show that there is an explicit bound such that the equation $x^p+y^p=z^3$ does not have a particular type of solution over $K=\Q(\sqrt{-d})$ where $d=1,7,19,43,67$ whenever $p$ is bigger than this bound.  During the course of the proof we prove various results about the irreducibility of Galois representations, image of inertia groups and  Bianchi newforms.	
	\end{abstract}

 \section{Introduction}

 Solving Diophantine equations is one of the oldest and widely studied topics in number theory. Yet we still don't have a general method that would allow us to produce solutions of a given Diophantine equation. Most of the time, it may be easier to show the nonexistence of solutions but even this can be quite challenging as it was the case for the proof of Fermat's Last Theorem. The method to solve the Fermat equation, used by Wiles in his famous proof, can be adapted to solve similar Fermat type equations. This strategy, which is referred as the `modular method', starts with an elliptic curve attached to a putative solution of the given equation. Then, using many celebrated theorems of the area,
  the problem can be reduced to  one of the following: computing newforms of a certain level, or
 computing all elliptic curves of a given conductor with particular information about torsion subgroup and rational isogeny, or 
 computing all solutions to an $S$-unit equation. Neither of these computations are easy in general. Especially if one needs to prove Fermat's theorem over a number field other than rationals, some fundamental theorems that go into the proof now become conjectures only, such as the modularity conjecture. Recently, there has been much progress in several different generalizations of this famous result.  For instance in  \cite{FS}, Freitas and Siksek 
 proved the asymptotic Fermat's Last Theorem (FLT) for certain totally real fields $K$. That is, they
 showed that there is a constant $B_K$ such that for any prime $p>B_K$, the only solutions to the 
 Fermat equation $a^p+b^p+c^p=0$ where $a,b,c \in \calO_{K}$ are the trivial ones satisfying $abc=0$.  
 Then, Deconinck \cite{HD} extended the results of Freitas and Siksek \cite{FS} to the generalized Fermat equation of the
 form $Aa^p+Bb^p+Cc^p=0$ where $A,B,C$ are odd integers belonging to a totally real field. Later in \cite{SS},
 \c{S}eng\"{u}n and Siksek proved the asymptotic FLT for any number field $K$ by assuming
 modularity.  This result has been generalized by Kara and Ozman in \cite{KO} to the case of generalized Fermat equation. Also, recently in \cite{Turcas2018} and \cite{Turcas2020} Turcas studied Fermat equation over imaginary quadratic fields $\Q(\sqrt{-d})$ with class number one.

 Similar generalizations are quite rare for other Fermat type equations such as  $x^p+y^q=z^r$. The solutions of this equation have been studied over rationals by many mathematicians including Darmon, Merel, Bennett and Poonen.  Several mathematicians have worked on similar Fermat type equations with different exponents over rational numbers. We have summarized these results in \cite{IKO} therefore will not repeat them here but we want to mention that not many results exist for generalizations of these to higher degree number fields. During the write-up of this paper, we have been informed about the work of Mocanu \cite{Mocanu} where she improves the results in \cite{IKO} and proves similar versions for the Diophantine equation of signature $(p,p,3)$. In this paper we also study the solutions of $x^p+y^p=z^3$ over number fields. However our results differ from Mocanu (and hence the results in \cite{IKO}) in the sense that we prove results about solutions of \ref{maineqn} over fields that are not totally real. In the appendix, we mention versions of these results for the Diophantine equation of signature $(p,p,2)$. Our results can be summarized as follows:

  \subsection{Our results}
   Let $K$ be a number field and $\calO_{K}$ be its ring of integers.  For a prime number $p$, we refer the equation
  \begin{equation}\label{maineqn}
  	a^p+b^p=c^3,\quad a,b,c \in \calO_{K}
  \end{equation}
  as \textit{the Fermat equation over K with signature (p,p,3)}.  A solution $(a,b,c)$ is
  called \textbf{trivial} if $abc=0$, otherwise \textbf{non-trivial}.   A solution $(a,b,c)$ of equation \eqref{maineqn} is called \textbf{primitive} if $a, b $ and $c$ are pairwise coprime. Since we will consider number fields with class number one, a putative solution $(a,b,c)$ can be scaled such that $a,b,c$ are coprime. Note also that if $a,b,c \in \mathcal{O}_K$ satisfy $a^p+b^p=c$ where $p>3$, then depending on what $p$ is modulo $3$, $(ac,bc,c^{\frac{p+1}{3}})$ or $(ac,bc,c^{\frac{p-1}{3}})$ is a non-primitive solution to the equation ($\ref{maineqn}$).%
  We consider only primitive solutions to \eqref{maineqn}. In \cite{DG}, it was shown that equation \eqref{maineqn} has finitely many primitive solutions.

  We say that ``\emph{the asymptotic Fermat Theorem holds for $K$ and signature $(p,p,3)$}'' if there is a constant
  $B_K$ such that for any prime $p>B_K$, the Fermat equation with signature $(p,p,3)$ (given in equation \eqref{maineqn}) does not have non-trivial, primitive solutions.

We have two main results about solutions of Equation \ref{maineqn}. The first result is an asymptotic result for some of the solutions over general number fields and the second one is an explicit result for general solutions over some  imaginary quadratic fields.

  \begin{theorem}\label{mainthm1} 
   Let $K$ be a number field with narrow class number $h_K^{+}=1$ satisfying Conjectures \ref{conj1} and \ref{conj2} and containing $\Q(\zeta_{3})$ where $\zeta_{3}$ is a primitive $3^{rd}$ root of unity. Assume that $\lambda$ is the only prime of $K$ lying above $3$. Let $W_K$ be the set of $(a,b,c)\in\calO_K$
   such that $(a,b,c)$ is a primitive solution to $x^p+y^p=z^3$ with $\lambda|b$. Then there is a constant $B_K$ -depending only on $K$- such that for $p > B_K$, equation \eqref{maineqn} has no solution $(a,b,c) \in W_K$. In this case we say that the asymptotic Fermat's Last Theorem holds for $W_K$ and signature $(p,p,3)$.
   \end{theorem}

\begin{theorem}\label{mainthm2}
		Let $K=\Q(\sqrt{-d})$ where $d\in\{1,7,19,43,67\}$ and $\ell_K$ be the largest prime in Table \ref{tab:primetorsion} corresponding to $K$. Let $C_K, M_K$ be defined as in Proposition \ref{irrCase2} and Corollary \ref{absirr}. Assume that Conjecture~\ref{conj1} holds true for $K$. Let $\lambda$ be the prime of $\calO_K$ lying over $3$. Then:

	\begin{itemize}
		\item[Case I] For any prime $p > {\rm max} \{\ell_K, C_K\}$, the Fermat equation over $K$ with signature $(p,p,3)$ does not have any non-trivial primitive solutions $(a,b,c)\in\calO_{K}$ such that $\lambda | b$.
		\item[Case II] If $p> B_K={\rm max}\{\ell_K,C_K,M_K\}$, $p$ splits in $K$ and $p \equiv 3 \pmod 4$ then the Fermat equation over $K$ with signature $(p,p,3)$ does not have any non-trivial primitive solutions $(a,b,c)\in\calO_{K}$.
	\end{itemize}
\end{theorem}

Various different techniques have to be combined to achieve these results. For the asymptotic result, we mostly follow the approach in \cite{IKO} and \cite{KO} which relies on the paper of Sengun and Siksek \cite{SS}. For instance, we need the absolute irreducibility of the associated Galois representation in order to apply the Serre's modularity conjecture. In order to do this, one needs to prove the irreducibility first and then pass to absolute irreducibility. This is rather classical when the Frey curve has  potentially multiplicative reduction at $q$ for some $q$ appearing in the denominator of the $j$-invariant of the Frey curve. However, for $(p,p,3)$ case the associated Frey elliptic curve has potentially good reduction when $3$ does not divide the norm of $ab$. Therefore, one can only get a result about solutions of particular type. This was also mentioned in the papers of Turcas \cite{Turcas2018} and \cite{Turcas2020}. It is sometimes possible to overcome this obstruction when working over explicit fields. For instance, as done by Najman and Turcas in \cite{NT}, using a result of Vaintrob and Larson, it is possible to prove the absolute irreducibility of the associated Galois representation when $p$ is bigger than an computable constant $B_K$. One can apply a similar argument to the Galois representation related to the Diophantine equation of signature $(p,p,3)$. Of course, in order to do this we need an irreducibility result for $\bar{\rho}_{E,p}$ when $p$ is bigger than an explicit constant $B$. This is a nontrivial task to do even in the case of the classical Fermat equation which was done by Freitas and Siksek in \cite{FSANT}. We combine all these to obtain our second result, which gives information about the solutions of $(p,p,3)$ over imaginary quadratic number fields of class number one.

  \subsection*{Acknowledgements}
  We are grateful to Samir Siksek for very helpful comments and discussions.
 
 \section{Preliminaries}
 
 In this section we give the necessary background to prove the results.  We follow  \cite{SS}, \cite{FS} and the references therein.
 \subsection{Conjectures}In this section we state the conjectures
 assumed in the above theorems. For more details we refer to Sections 2 and 3 of \cite{SS}.

 Let $K$ be a number field with the ring of integers $\calO_{K}$ and $\frakN$ be an ideal of $\calO_{K}$.  The following result is proved by \c{S}eng\"{u}n and Siksek in \cite{SS}.
 
 \begin{prop}[\cite{SS}, Proposition 2.1]\label{eigenform}
 	There is an integer $B(\frakN)$, depending only on $\frakN$, such that for any prime $p>B(\frakN)$, every weight two, mod $p$ eigenform of level $\frakN$ lifts to a complex one.
 \end{prop}

 In order to run the modular approach to solve Diophantine equations one needs generalized modularity theorems.
 Due to the lack of their existence we can only prove our theorems up to some conjectures. One of the assumed conjectures is a special case of Serre's modularity conjecture over number fields, stated below:
 
 \begin{conj}[\cite{FKS}, Conjecture 4.1]\label{conj1}  Let $\overline{\rho}:G_K\rightarrow GL_2(\overline{\mathbb{F}}_p)$ be an odd, 
 	irreducible, continuous representation with 
 	Serre conductor $\frakN$ (prime-to-p part of its Artin conductor) 
 	and such that $\det(\overline{\rho})=\chi_p$ is the mod p cyclotomic character. 
 	Assume that $p$ is unramified in $K$ and that $\overline{\rho}|_{G_{K_{\frakp}}}$ arises from
 	a finite-flat group scheme over $\calO_{K_{\frakp}}$ for every prime $\frakp|p$.  Then there is a weight 2 mod $p$
 	eigenform $\theta$ over $K$ of level $\frakN$ such that for all primes $\frakq$ coprime to $p\frakN$, we have
 	\[
 	\tr(\overline{\rho}(\Frob_{\frakq}))=\theta(T_{\frakq}),
 	\]
 	where $T_{\frakq}$ denotes the Hecke operator at $\frakq$.
 \end{conj}
 
 Additionally, we will use a special case of a fundamental conjecture from Langlands Programme for the asymptotic result. Note that we don't need Conjecture \ref{conj2} for Theorem \ref{mainthm2}.
 
 \begin{conj}[\cite{SS}, Conjecture 4.1]\label{conj2}
 	Let $\frakf$ be a weight 2 complex eigenform over $K$ of level $\frakN$ that is non-trivial and new.  If $K$ has some
 	real place, then there exists an elliptic curve $E_{\frakf}/K$ of conductor $\frakN$ such that 
 	\begin{equation}\label{c2eqn}
 	\#E_{\frakf}(\calO_K/\frakq)=1+\nr(\frakq)-\frakf(T_{\frakq})\quad\mbox{for all}\quad\frakq\;\nmid\;\frakN.
 	\end{equation}
 	If $K$ is totally complex, then there exists either an elliptic curve $E_{\frakf}$ of conductor $\frakN$ satisfying (\ref{c2eqn})
 	or a fake elliptic curve 
 	$A_{\frakf}/K$, of conductor $\frakN^2$, such that
 	\begin{equation}
 	\#A_{\frakf}(\calO_K/\frakq)=(1+\nr(\frakq)-\frakf(T_{\frakq}))^2\quad\mbox{for all}\quad\frakq\;\nmid\;\frakN.
 	\end{equation}
 \end{conj}

 \subsection{Frey curve and related facts} In this section we collect some facts related to the Frey curve associated to a putative solution of the equation 
 \eqref{maineqn} and the associated Galois representation.
 
Let $G_{K} $ be the absolute Galois group of a number field $K$, let $E/K$ be an elliptic curve and $ \overline{\rho}_{E,p}$ denote the mod $p$ Galois representation of $E$.  We use $\frakq$ for an arbitrary prime of $K$, and $G_\frakq$ and $I_\frakq$ respectively for the decomposition and inertia subgroups of $G_K$ at $\frakq$.  For a putative solution $(a,b,c)$ to the equation \eqref{maineqn} with a prime exponent $p$, we associate the Frey elliptic curve as in \cite{BVY},
 \begin{equation} \label{Frey}
 E=E_{a,b,c}: Y^2+3cXY+b^pY=X^3
 \end{equation}
 whose arithmetic invariants are given by $\Delta_E=3^3(ab^3)^p$, $\displaystyle j_E=\frac{3^3c^3(9a^p+b^p)^3}{(ab^3)^p}$ and $c_4(E)=9c(9a^p+b^p), c_6(E)=-3^3(3^3c^6-2^23^2c^3b^p+2^3b^{2p})$.

 For the result below, we have the same assumptions on the number field $K$ as explained in the introduction. Namely, $K$ is a  number field of degree $d$ such that the narrow class number of $K$, $h_K^{+}=1$ and there is a unique prime $\lambda $ over $3$. 
\begin{lemma}\label{semist} 
Let $\lambda^e=3 \mathcal{O}_K$ where $\mathcal{O}_K$ is the integer ring of $K$ .
The Frey curve $E$ is semistable away from $\lambda$ and has a $K$-rational point of order $3$. 
	The determinant of $\overline{\rho}_{E,p}$ is the mod $p$ cyclotomic character.  The Galois representation $\overline{\rho}_{E,p}$ is finite flat at every prime $\frakp$ of $K$ that lies above $p$. 
	Moreover, the conductor $\mathcal{N}_E$ attached to the Frey curve $E$ is given by
	$$\mathcal{N}_E=\lambda^\epsilon\prod_{\mathfrak{q}|ab, \mathfrak{q} \nmid 3}\mathfrak{q},$$ 
	where 
	\begin{enumerate}[(i)]
		\item $\epsilon=0,1$ if $\lambda | ab$ and $p > 2e$.
		\item $\epsilon \geq 2,3$ if $\lambda \nmid ab$. Moreover if $e=1$, $\epsilon = 2,3.$  
	\end{enumerate}
	
	In particular,if $\lambda | ab$  the curve $E$ is semistable, and otherwise, the curve $E$ has additive reduction at $\lambda$.
	
	The Serre conductor $\frakN_E$ of $\overline{\rho}_{E,p}$ is supported on
	$\lambda$ and belongs to a finite set depending only on the
	field $K$.   
\end{lemma}

\begin{proof}
	Assume that the narrow class number $h_K^{+}=1$.  Recall that the invariants $c_4(E), c_6(E)$ and $\Delta_E$ of the model $E$ are given by
	$$c_4(E)= 9c(9a^p+b^p), \; c_6(E)=-3^3(3^3c^6-2^23^2c^3b^p+2^3b^{2p}) \; \Delta_E=3^3(ab^3)^p.$$

	Suppose $\frakq \neq \lambda$ divides $\Delta_E$, which implies that $ab$ is divisible by $\frakq$. Since $a,b,c$ are pairwise coprime, $\frakq$ divides either $a$ or $b$. Therefore, $c_4(E)= 9c(9a^p+b^p)$ is not divisible by $\frakq$, i.e. $v_{\frakq}(c_4(E))=0$. Hence, the given model is minimal and $E$ is semistable at $\frakq$. Moreover, we have $p|v_{\frakq}(\Delta_E)$. It follows from \cite{Serre} that $\overline{\rho}_{E,p}$ is finite flat at $\frakq$ if $\frakq$ lies above $p$. We can also deduce that $\overline{\rho}_{E,p}$ is unramified at $\frakq$ if $\frakq \nmid p$.

	Now assume that  $\lambda$ divides $ab$. Note that  $\lambda$  can only divide one of $a$ or $b$. Without loss of generality say $\lambda | b$. The result regarding $\lambda \nmid ab$ can be handled in an identical manner. 
	
	In all cases, the valuation $v_\lambda(\mathcal{N}_E)=\epsilon$ can be calculated via \cite[Tableau~III]{pap}. Note that when $\lambda | ab$, the equation is not minimal. After using the change of variables $X=3^2x, Y=3^3y$ we get $v_\lambda(c_4(E))=v_\lambda(c_6(E))=0$ when $p> 2e$, hence $\epsilon=0,1.$ 
	 
	The statement concerning the determinant is a well known consequence of the Weil pairing attached to elliptic curves. The fact that the Frey curve $E$ has a $K$-rational point of order $3$ follows from \cite[Lemma~2.1 (c)]{BVY}.
	Finally, to show that there can be only finitely many Serre conductors $\frakN_E$ note that only the prime
		$\lambda$ can divide $\frakN_E$.  As $\frakN_E$ divides the conductor $\calN_E$ of $E$, 
		$v_{\lambda}(\frakN_E)\leq v_{\lambda}(\calN_E)\leq 2+3 v_{\lambda}(3)+6 v_{\lambda}(2)$ by \cite[Theorem IV.10.4]{Sil94}.
		Hence, there can be only finitely many Serre conductors and they only depend on $K$.  
		
\end{proof}

Given a number field $K$, we obtain a \emph{complex conjugation} for every real embedding 
$\sigma: K \hookrightarrow \mathbb R$ and every extension $\tilde{\sigma}: \overline{K} \hookrightarrow \mathbb C$ of 
$\sigma$ as $\tilde{\sigma}^{-1}\iota \tilde{\sigma} \in G_K$ where $\iota$ is the usual complex conjugation.
Recall that a 
representation $\overline{\rho}_{E,p}: G_K \rightarrow \GL_2(\overline{\mathbb F}_p)$ is \emph{odd} if the determinant of every complex 
conjugation is $-1$.  If the number 
field $K$ has no real embeddings, then we immediately say that $\overline{\rho}_{E,p}$ is odd.

The following results give us information about the image of inertia groups under the Galois representation $\overline{\rho}_{E,p}$.

\begin{lemma}[\cite{FS}, Lemma 3.4]\label{image of inertia}
	Let $E$ be an elliptic curve over $K$ with $j$-invariant $j_E$.  Let $p \geq 5$ and $\mathfrak{q}\nmid p$ be a prime  
	of $K$. Then $p|\#\overline{\rho}_{E,p}(I_{\mathfrak{q}})$ if and only if $E$ has potentially multiplicative
	reduction at $\mathfrak{q}$ (i.e. $\upsilon_{\mathfrak{q}}(j_E)<0$) and $p \nmid \upsilon_{\mathfrak{q}}(j_E)$.
\end{lemma}   
By using the previous result we obtain:

\begin{lemma}\label{pot.mult.red.}
	Let $\lambda$ be the only prime ideal of $K$ lying above $3$  and let $(a,b,c)\in W_K$ with prime exponent $p>3\upsilon_{\lambda}(3)$. Let $E$ be the Frey curve in \eqref{Frey} and write $j_E$ for its $j$-invariant.  
	Then $E$ has potentially multiplicative reduction at $\lambda$ and $p|\#\overline{\rho}_{E,p}(I_{\lambda})$ where $I_{\lambda}$ denotes an inertia subgroup of $G_{K}$ at $\lambda$.
\end{lemma}  

\begin{proof}
	Assume that $\lambda$ is the only prime ideal of $K$ lying above $3$ with $\upsilon_{\lambda}(b)=k$. Then 
	$\upsilon_{\lambda}(j_E)=3\upsilon_{\lambda}(3)-3pk$. Since 
	$p>3\upsilon_{\lambda}(3)$, we have $v_{\lambda}(j_E)<0$ and clearly $p \nmid \upsilon_{\lambda}(j_E)$. This implies that
	$E$ has potentially multiplicative reduction at $\lambda$ and $p| \#\overline{\rho}_{E,p}(I_{\lambda})$. 
\end{proof} 

The following well-known result about subgroups of $\GL_2(\mathbb F_p)$ will be frequently used.

\begin{theorem}\label{subgroups} Let $E$ be an  elliptic curve over a number field $K$ of degree $d$ and let $G \leq \GL_2(\mathbb F_p)$ be the
	image of the mod $p$ Galois representation of $E$.
Then the following holds:
	\begin{itemize}
		\item if $p | |G|$ then either $G$ is reducible or $G$ contains $\SL_2(\mathbb F_p)$ hence absolutely irreducible. 
		\item if $p \nmid |G|$ and $p > 15 d +1$ then $G$ is contained in a Cartan subgroup or $G$ is contained in the normalizer of Cartan subgroup but not the Cartan subgroup itself.
	
	\end{itemize}
\end{theorem}

\begin{proof}
	For the proof main reference is \cite[Lemma 2]{SD}. The version above including the proof of the second part is from \cite{localglobal} Propositions 2.3  and 2.6.
	\end{proof}

\section{Properties of Galois representations}

\subsection{Level Reduction}

In this section we will be relating the Galois representation attached to Frey curve with another representation of lower 
level.

\begin{theorem}\label{thm:levelred}
	Let $K$ be a number field with $h^{+}_K=1$ satisfying Conjecture \ref{conj1} and Conjecture \ref{conj2}.  Assume that $\lambda$ is the only prime of $K$ above $3$. Then there is a 
	constant $B_K$ depending only on $K$ such that the following holds. Let $(a,b,c)\in W_K$ be a non-trivial solution to the equation \eqref{maineqn} with exponent $p>B_K$. Let $E/K$ be the associated Frey curve defined in \eqref{Frey}. Then there is an
	elliptic curve $E'/K$ such that the following statements hold:
	
	\begin{enumerate}[(i)]
		\item $E'$ has good reduction away from $\lambda$.
		\item $E'$ has a $K$-rational point of order $3$.
		\item $\overline{\rho}_{E,p}\sim\overline{\rho}_{E',p}$.
		\item $v_{\lambda}(j')<0$ where $j'$ is the $j$-invariant of $E'$.
	\end{enumerate}
\end{theorem}
We will give the proof of this Theorem in Section \ref{appcon} after stating the necessary lemmas. The following is Proposition 6.1 from \c{S}eng\"{u}n and Siksek \cite{SS}. We include its statement for the convenience of the reader but we will omit its proof and refer to \cite{SS} instead.
\begin{proposition}\label{irred}
	Let $L$ be a Galois number field and let $\frakq$ be a prime of $L$. There is a constant $B_{L,\frakq}$ such that 
	the following is true. Let $p > B_{L, \frakq}$ be  a rational prime. Let $E/L$ be an elliptic curve that is semistable 
	at all $\frakp | p$ and having potentially multiplicative reduction at $\frakq$. Then $\overline{\rho}_{E,p}$ is irreducible. 
	
\end{proposition}

By applying the above proposition to the Frey curve, we get the following corollary.

\begin{cor}\label{galrepsur}
	Let $K$ be a number field with $h^{+}_K=1$, and suppose that $\lambda$ is the only prime of $K$ above $3$.  There is a constant $C_K$ such that if $p>C_K$ and 
	$(a,b,c)\in W_K$ is a non-trivial solution to the Fermat equation with signature $(p,p,3)$, then  $\overline{\rho}_{E,p}$ is surjective, where $E$ is the Frey curve given in \eqref{Frey}.
	\end{cor}

\begin{proof}
	By Lemma \ref{pot.mult.red.}, $E$ has potentially multiplicative reduction at $\lambda$.  Also, $E$ is semistable
	away from $\lambda$ from Lemma \ref{semist}.  Let $L$ be the Galois closure of $K$, and let $\frakq$ be a prime
	of $L$ above $\lambda$.  Now, by applying Proposition \ref{irred}, we get a constant $B_{L,\frakq}$ such that 
	$\overline{\rho}_{E,p}(G_L)$ is irreducible whenever $p>B_{L,\frakq}$.  Note that there are only finitely many choices of $\frakq|3$ in $L$ and $L$ only depends on $K$.  Hence, we can obtain a constant depending only on $K$ and we denote it by $C_K$.
	If necessary, enlarge $C_{K}$ so that $C_{K}>3v_{\lambda}(3)$.  Now, we apply Lemma \ref{pot.mult.red.} and see that
	the image of $\overline{\rho}_{E,p}$ contains an element of order $p$.  By Theorem \ref{subgroups} any subgroup of $\GL_2(\F_p)$ 
	having an element of order $p$ is either reducible or contains $\SL_2(\F_p)$. As $p>C_K>3v_{\lambda}(3)$, the image
	contains $\SL_2(\F_p)$.  Finally, we can ssume that $K\cap\Q(\zeta_p)=\Q$ by taking $C_K$ large enough if needed.
	Hence, $\chi_p=\det(\overline{\rho}_{E,p})$ is surjective.
\end{proof}
\subsubsection{\bfseries{Proof of Theorem \ref{thm:levelred}}}
\label{appcon}

	In this subsection, Theorem \ref{thm:levelred} will be proven.  The proof closely follows the ideas in \cite{KO}, however we will give it here for the sake of completeness and for the convenience of the reader.  We continue with the notations introduced in the statement of
	Theorem \ref{thm:levelred} and the assumptions of the theorem.
	\begin{lemma}\label{reduced}
		There is a non-trivial, new (weight 2) complex eigenform $\frakf$ which has an associated elliptic curve 
		$E_{\frakf}/K$ of conductor $\frakN'$ dividing $\frakN_E$.
	\end{lemma}
	\begin{proof}
		We first show the existence of such an eigenform $\frakf$ of level $\frakN_E$ supported only on $\{\lambda\}$.
		
		By Corollary \ref{galrepsur}, the representation $\overline{\rho}_{E,p}:G_K\rightarrow GL_{2}(\mathbb{F}_p) $ is surjective 
		hence is absolutely irreducible for $p>C_K$.  Now, we apply  Conjecture \ref{conj1} to deduce that there is 
		a weight $2 \mod p$ eigenform $\theta$ 
		over $K$ of level $\frakN_E$, with $\frakN_E$ as in Lemma \ref{semist}, such that for all primes $\frakq$ coprime to $p\frakN$, we have 
		\[\tr(\overline{\rho}_{E,p}(\Frob_{\frakq}))=\theta(T_{\frakq}).\]
		We also know from the same lemma that there are only finitely many possible levels $\frakN$.  Thus by taking $p$
		large enough, see Proposition \ref{eigenform}, for any level $\frakN$
		there is a weight 2 complex eigenform $\frakf$ with level $\frakN$ which is a lift of $\theta$.
		Note that since there are only finitely many such eigenforms $\frakf$
		and they depend only on $K$, from now on we can suppose that every constant depending on these eigenforms depends
		only on $K$.
		
		Next, we recall that if $\Q_{\frakf}\neq\Q$ then  there is a constant $C_{\frakf}$ depending only on $\frakf$ such that 
		$p<C_{\frakf}$ [\cite{SS}, Lemma 7.2].  Therefore, by taking $p$ sufficiently large, we assume that $\Q_{\frakf}=\Q$. 
		In order to apply Conjecture \ref{conj2}, we need to
		show that $\frakf$ is non-trivial and new.  As $\overline{\rho}_{E,p}$ is irreducible, the eigenform $\frakf$ is non-trivial.  
		If $\frakf$ is new, we are done.  If not, we can replace it with an equivalent new eigenform of smaller level.  Therefore,
		we can take $\frakf$ new with level $\frakN'$ dividing $\frakN_E$.
		Finally, we apply Conjecture \ref{conj2} and obtain that $\frakf$ either has an associated elliptic curve $E_{\frakf}/K$ 
		of conductor $\frakN'$, or has an associated fake elliptic curve $A_{\frakf}/K$ of conductor $\frakN_E^2$.
		
		By Lemma \ref{fake} below, if $p>24$, then $\frakf$ has an associated elliptic curve $E_{\frakf}$. As a result, we can assume that $\overline{\rho}_{E,p}\sim \overline{\rho}_{E',p}$ where $E'=E_{\frakf}$ is an elliptic
		curve with conductor $\frakN'$ dividing $\frakN_E$.  

	\end{proof}
	\begin{lemma}[\cite{SS}, Lemma 7.3]\label{fake}
		If $p>24$, then $\frakf$ has an associated elliptic curve $E_{\frakf}$.
	\end{lemma}

	We can now give the proof of Theorem \ref{thm:levelred}.\\

	\begin{proof3}
		Lemma \ref{fake} gives us that if $p>24$, then $\frakf$ has an associated elliptic curve $E_{\frakf}$.
		Therefore, by Lemma \ref{reduced} we can assume that $\overline{\rho}_{E,p}\sim\overline{\rho}_{E',p}$ where $E'=E_{\frakf}$ is an
		elliptic curve
		of conductor $\frakN'$ dividing $\frakN_E$.
		
		\begin{lemma}\label{3-torsion}
			If $E'$ does not have a non-trivial $K$-rational point of order $3$ and is not isogenous to an elliptic curve with a non-trivial $K$-rational point of order $3$ then $p<C_{E'}$.
		\end{lemma}
		
		\begin{proof}
			By Theorem \ref{torsion point} there are infinitely many primes $\frakq$ such that $\# E'(\F_\frakq)\not \equiv 0 \pmod 3$.  Fix such a prime $\frakq\neq \lambda$ and note that $E$ is semistable at $\frakq$.  If $E$ has good reduction at $\frakq$, then $\# E(\F_\frakq)\equiv \# E'(\F_\frakq) \pmod p$.  Since $3|\# E(\F_\frakq)$ the difference, which is divisible by $p$, is nonzero.  As the difference belongs to a finite set depending on $\frakq$, $p$ becomes bounded.  If $E$ has multiplicative reduction at $\frakq$, we obtain 
			\[
			\pm(\Norm(\frakq)+1)\equiv a_{\frakq}(E') \pmod p
			\]
			by comparing the traces of Frobenius.  We see that this difference being also nonzero and depending only on $\frakq$ gives a bound for $p$.
		\end{proof}

		Now, suppose $E'$ is $3$-isogenous to an elliptic curve $E''$.  As the isogeny induces an isomorphism
		$E'[p]\cong E''[p]$ of Galois modules ($p\neq 3$), we get $\overline{\rho}_{E,p}\sim \overline{\rho}_{E',p}\sim\overline{\rho}_{E'',p}$ 
		completing the proof of (iii).  After possibly replacing $E'$ by $E''$, we can suppose that $E'$ has a $K$-rational point of order $3$
		giving us (ii).
		
		It remains to prove  $v_{\lambda}(j')<0$ where $j'$ is the $j$-invariant of $E'$.   By Lemma 5.2 in \cite{SS}, $p$ divides the size of $\overline{\rho}_{E,p}(I_{\lambda})$. 
		Now, Lemma 5.1 in \cite{SS} implies that $v_{\lambda}(j')<0$ since the sizes of  
		$\overline{\rho}_{E,p}(I_{\lambda})$ and $\overline{\rho}_{E',p}(I_{\lambda})$ are equal.  
		
	\end{proof3}
	
	The following theorem of Katz is used in the proof of the above lemma.

	\begin{theorem}[\cite{Katz},Theorem 2]\label{torsion point}
		Let $E$ be an elliptic curve over a number field $K$, and $m\geq 2$ an integer. For each prime $\mathfrak{p}$ of $K$ at which $E$ has good reduction let $N(\mathfrak{p})$ denote the number of $\mathbb{F}_{\mathfrak{p}}$-rational points on $E\mod{\mathfrak{p}}$. If we have 
		$$N(\mathfrak{p}) \equiv 0 \pmod{m}$$
		for a set of primes $\mathfrak{p}$ of density one in $K$, then there exists a $K$-isogenous elliptic curve $E'$ defined over $K$ such that $$\#(\Tors E'(K)) \equiv 0 \pmod{m}.$$
\end{theorem}

 \subsection{Irreducibility of Galois Representations}

 Throughout this section $K=\Q(\sqrt{-d})$, where $d \in \{1,7,19,43,67\}$, $(a,b,c) \in \mathcal O_K$ is a nontrivial, primitive, putative solution of the equation $x^p+y^p=z^3$. 

 The idea of this section is to prove that when $p$ is bigger than an explicit constant then the mod $p$ Galois representation $\overline{\rho}_{E,p}:G_\Q \rightarrow \Aut(E[p]) \cong \GL_2(\F_p)$ attached to $E$ is absolutely irreducible.
 
 We will use the following result of Freitas and Siksek:
 
 \begin{lemma}[\cite{FSANT}, Lemma 6.3]\label{FSlemma} Let $E$ be an elliptic curve over a number field $K$ with conductor $\mathcal{N}_E$, let $p$ be a prime $>5$. 	Suppose that $\rho_p=\overline{\rho}_{E,p}$ is reducible. Write $\rho_p \sim  \begin{pmatrix}
 		\theta     & * \\
 		0      & \theta' 
 	\end{pmatrix} $ where $\theta, \theta': G_{K} \rightarrow \mathbb F_p^*$ are the \emph{isogeny characters.} Let $\mathcal N_{\theta}, \mathcal N_{\theta'}$ denote the conductors of these characters. Fix a prime $\mathfrak q \nmid p$ of $\mathcal O_K.$ 
 	
 	We have the following:
 	\begin{itemize}
 		\item if $E$ has good or multiplicative reduction at $\mathfrak q$, then $v_{\mathfrak q}(\mathcal N_{\theta})=v_{\mathfrak q}(\mathcal N_{\theta'})=0$
 		\item if $E$ has additive reduction at $\mathfrak q$, then $v_{\mathfrak q}(\mathcal N_E)$ is even and $v_{\mathfrak q}(\mathcal N_{\theta})=v_{\mathfrak q}(\mathcal N_{\theta'})=v_{\mathfrak q}(\mathcal N_E)/2.$
 	\end{itemize}
 \end{lemma}

 \begin{prop}\label{irrCase2}
 	Let $E/K$ be the Frey curve attached to a putative solution to the equation (\ref{maineqn}) and $p > C_K$ be a prime where $C_K$ is defined as below. Then $\overline{\rho}_{E,p}$  is irreducible. 
 	
 	$C_K=47$ if the equation $y^3+24b^pcy+16b^{2p} \equiv 0 \pmod {\lambda^2}$ has a solution in the ring of integers of the local field $K_\lambda$, otherwise $C_K=44483$.
 \end{prop}

 \begin{proof}
 	Suppose that $\rho_p=\overline{\rho}_{E,p}$ is reducible. Write $\rho_p \sim  \begin{pmatrix}
 		\theta     & * \\
 		0      & \theta' 
 	\end{pmatrix} $ where $\theta, \theta': G_{K} \rightarrow \mathbb F_p^*$ are the isogeny characters. Let $\mathcal N_{\theta}, \mathcal N_{\theta'}$ denote the conductors of these characters. Fix a prime $\mathfrak q \nmid 3p$ of $\mathcal O_K.$ By \cite[Lemma 6.3]{FSANT} we get $v_{\mathfrak q}(\mathcal N_{\theta})=v_{\mathfrak q}(\mathcal N_{\theta'})=0$ since $E$ is semistable away from $\mathfrak P$. 
 	
 	By Lemma \ref{FSlemma} and Lemma  \ref{semist}, $v_{\lambda}(\mathcal N_{\theta})=v_{\lambda }(\mathcal N_{\theta'}) \in \{0,1\}$.
 	
 	Now we deal with $\frakp | p.$ 
 	\begin{enumerate}
 		\item Say $\mathcal N_{\theta}$ or $\mathcal N_{\theta'}$ is relatively prime to $p$. Note that interchanging $\theta$ and $\theta'$ corresponds to replacing $E$ with an isogenous elliptic curve $E/\ker \theta$. Since $\ker \theta$ is an order $p$, $K$-rational subgroup of $E[p]$, $E$ and $E/\ker \theta$ are $p$-isogenous. Therefore, without loss of generality assume that $(p,\mathcal N_{\theta})=1$ and $v_{\frakp}(\mathcal N_{\theta})=0$ for all $\frakp | p$ as in the previous case. We also have $v_{\mathfrak q}(\mathcal N_{\theta})=0$ for all $\mathfrak q \nmid 3p$, as explained above. Therefor $\mathcal N_{\theta}=\lambda^m$ where $m=0$ or $1$ which implies that $\theta$ is a character of the ray class group of modulus $\lambda^m$ of $K$.
 		
 		Using \texttt{Magma} we computed the ray class groups for these moduli and get the following groups only:
 		
 		$$ \{1\}, \Z/2\Z, \Z/4\Z .$$

 		\begin{itemize}
 			\item If the order of $\theta$ is one then $\theta$ is trivial. Then  $\rho_p \sim  \begin{pmatrix}
 				1     & * \\
 				0      & \theta' 
 			\end{pmatrix} $ and this implies that $E$ has a $K$-rational point of order $p$. By Lemma~\ref{semist} $E$ has also a point of order $3$, then $E(K)$ has a $3p$-torsion point but by the work of  Kamienney, Kenku and Momose (\cite{Kam}, \cite{KM}) this is not possible when $p \geq 7$, hence contradiction.
 			\item If the order of $\theta$ is two then we can conclude that $E$ has a point of order $3p$ over a quadratic extension $L$ of $K$ and get a contradiction since by \cite{der2021torsion}, $E(L)[p]=\{0\}$ if $p>17$. Here $L$ is the number field cut out by the character $\theta^2$ i.e. $|L:\Q|=4.$
 			\item If the order of $\theta$ is four, let $L$ be the unique quadratic extension of $K$ cut out by $\theta^2$. Then $\theta_{G_L}$ is quadratic. Let $E'$ be the twist of $E$ by $\theta_{G_L}$. The elliptic curve $E'$ is also over $L$ and has a point of order $p$ as in the previous case. Again, we get a contradiction by \cite{der2021torsion}.
 			
 		\end{itemize}
 		
 		\item Now we are left with the case that neither $\mathcal N_{\theta}$ nor $\mathcal N_{\theta'}$ is relatively prime to $p$. Recall that $E$ is semistable away from $\lambda$ and $p$ is not ramified in $K$. Then $p$ is either inert or $p$ splits in $K$.
 		\begin{enumerate}
 			\item $p$ is inert in $K$:  By \cite[Corollary 6.2]{FSANT} $E$ cannot have good supersingular reduction at $p$. Therefore $E$ has good ordinary or multiplicative reduction at $p$ and  \begin{equation}\label{eq:inertia} 
 				\rho_p| I_p \sim  \begin{pmatrix}
 					\chi_p     & * \\
 					0      & 1 
 			\end{pmatrix} \end{equation} see \cite[Section 3]{ECbook} where $\chi_p$ is the mod $p$ cyclotomic character. This shows that one of $\theta, \theta'$ is trivial. By changing $E$ with an isogenous elliptic curve we can assume that $\theta$ is trivial and as above this implies that $E(K)[3p] \neq 0$, which is a contradiction since $p>7$.
 			\item $p$ splits in $K$: Say $p=\mathcal P \mathcal P'$ and $\mathcal P | \mathcal N_{\theta}, \mathcal P' \nmid \mathcal N_{\theta}$ and $\mathcal P \nmid \mathcal N_{\theta'}, \mathcal P' | \mathcal N_{\theta'}$. By \cite[Corollary 6.2]{FSANT} we know that $E$ has good ordinary or multiplicative reduction at $\mathcal P$ and $ \mathcal P'$ by Equation \ref{eq:inertia}, we see that one of the characters $\theta, \theta'$ is ramified at $\mathcal P$ and the other is ramified at $\mathcal P'$, $\theta| I_{\mathcal P} = \chi_p|_{\mathcal P}$ and $\theta'| I_{\mathcal P'} = \chi_p|_{\mathcal P'}$. Hence $\theta$ is unramified away from $\mathcal P$ and $ \lambda$ since all bad places of $E$ except possibly $\lambda$ are of potentially multiplicative reduction.

 			\begin{enumerate}
 				\item $\lambda$ divides $ab$: In this case by Lemma \ref{semist}, $E$ has multiplicative or good reduction at $\lambda$; therefore, we can say that $\theta$ is unramified away from $\mathcal P$. The character $\theta^2|_{I_\mathcal P}=\chi_p^2|_{I_\mathcal P}$ is also unramified away from $\mathcal P$ therefore by \cite[Lemma 4.3]{Turcas2018}, $\theta(\sigma_\lambda) \equiv \text{Norm}_{K_\mathcal P/\Q_p}(\alpha)^2 \pmod p$ where $\sigma_\lambda$ is the Frobenius automorphism at $\lambda= \langle 3 \rangle $. We also know that by \cite[Lemma 6.3]{SS} $\theta^2(\sigma_\lambda) \equiv 1 \pmod p$ (note that $E$ has multiplicative reduction at $\lambda$.) Therefore, we have $p |  \text{Norm}_{K_\mathcal P}/\Q_p(\lambda)^2 - 1$, contradiction since $p>20$.
 				
 				\item $\lambda$ does not divide $ab$:  In this case by Lemma \ref{semist}, $E$ has additive reduction at $\lambda$ so the above argument fails. Recall that $v_\lambda(\Delta_E)=v_{\lambda }(3^3b^{3p}a^p)=3$ and $v_\lambda (c_6(E))=3$. By \cite[Page 356]{Kraus90} we see that $\theta^4$ or $\theta^{12}$ is unramified at $\lambda$. The case $\theta^{12}$ happens when  the equation $y^3+24b^pcy+16b^{2p} \equiv 0 \pmod \lambda^2$ has a solution. 
 				Therefore we have  $\theta^4|_{I_\mathcal P}=\chi_p^4|_{I_\mathcal P}$ is  unramified away from $\mathcal P$ or $\theta^{12}|_{I_\mathcal P}=\chi_p^{12}|_{I_\mathcal P}$ is  unramified away from $\mathcal P$.
 				
 				By  \cite[Lemma 4.3]{Turcas2018}, $\theta^4(\sigma_\mathfrak P) \equiv \text{Norm}_{K_\mathcal P/\Q_p}(\alpha)^2=9 \pmod p$ where $\sigma_\mathfrak P$ is the Frobenius automorphism at $\mathfrak P=\langle \alpha \rangle$. Therefore the polynomial $x^4-9$ has a root $\theta(\sigma_\mathfrak P)$ modulo $p$. Recall that $E$ has potentially good reduction at $\lambda$ since $v_\lambda(j_E) >0$. Let $P_\lambda (x)$ be the characteristic polynomial of the Frobenius of $E$ at $\lambda$. Since $\rho_p \sim  \begin{pmatrix}
 					\theta     & * \\
 					0      & \theta' 
 				\end{pmatrix} $ we get $P_\lambda(x) \equiv (x-\theta(\sigma_\lambda ))(x-\theta'(\sigma_\lambda)) \pmod p$. 
 				
 				Hence, we can conclude that $p | \Res(x^4-9, P_\lambda (x))$ where $\Res$ denotes the resultant of the polynomials.  Note that $P_\lambda  (x) \in \mathbb Z[x]$ and its roots have absolute value less than or equal to $\sqrt{\Norm(\lambda)}=3.$ (see \cite[ Proposition 1.6]{David2012}). Then the possibilities for $P_\lambda (x)$ are as follows:
 				
 				$$ P_1(x)=x^2+9, P_2(x)=x^2-x+9, P_3(x)=x^2-2x+9, P_4(x)=x^2+x+9, P_5(x)=x^2+2x+9$$
 				
 					$$ P_6(x)=x^2+3x+9, P_7(x)=x^2-3x+9 ,  P_8(x)=x^2+4x+9, P_9(x)=x^2-4x+9 $$
 					
 					$$  P_{10}(x)=x^2+5x+9, P_{11}(x)=x^2-5x+9, P_{12}(x)=x^2+6x+9, P_{13}(x)=x^2-6x+9   $$
 				
 				We computed these 13 resultants and none of them has a prime divisor greater than $ 47 $. Since $p >47$, we get a contradiction.
 				
 				For the case of $\theta^{12}|_{I_\mathcal P}=\chi_p^{12}|_{I_\mathcal P}$, similarly we need to compute the resultants of $P_i$ with $x^{12}-9$ and see that none of them has a prime divisor grater than $44483$.
 			
 			\end{enumerate}
 			
 		\end{enumerate}
 	\end{enumerate}
  \end{proof}

 	In this corollary we will summarize the cases where we have absolute irreducibility.

  \begin{cor}\label{absirr} \begin{enumerate}
  		\item Let $p$ be odd. If $\lambda | ab$ and $\bar{\rho}_{E,p}$ is irreducible then $\overline{\rho}_{E,p}$ is absolutely irreducible. 
  		\item Assume $\lambda \nmid ab$. Let $K$ be an imaginary quadratic field and $p > M_K$ for some effectively computable constant $M_K$ depending only on $K$. 
  		\begin{enumerate}
  			\item Say $p$ splits in $K$ and $p \equiv 3 \pmod 4$.  If $\overline{\rho}_{E,p}$ is irreducible then $\overline{\rho}_{E,p}$ is absolutely irreducible. 
  				\item Say $p \equiv 1\pmod 3$ and $\overline{\rho}_{E,p}(I_\lambda)$ is divisible by $3$ (which means the inertia has order $12$ and this happens if and only if $y^3+24b^pcy+16b^{2p} \equiv 0 \pmod {\lambda^2}$ doesn't have a solution). If $\overline{\rho}_{E,p}$ is irreducible then $\overline{\rho}_{E,p}$ is absolutely irreducible.
  		\end{enumerate}
  	
  		\item Let $p$ be odd. If $K$ is totally real then $\overline{\rho}_{E,p}$ is irreducible if and only if it is absolutely irreducible. 
  	\end{enumerate}
  	
  	\end{cor}

  \begin{proof} 
  		\begin{enumerate}
  		\item 
   The Frey curve $E$ attached to a putative primitive solution $(a,b,c)\in \mathcal O_K^3$  of $x^p+y^p=z^3$ is semistable when $\lambda | ab$ where $\lambda$ is the prime of $\mathcal O_K$ lying over $3$. These were discussed in Lemma \ref{semist}. By Proposition \ref{irrCase2} we know that $\overline{\rho}_{E,p}$ is irreducible when $p$ is big enough. In Proposition \ref{irrCase2} we make this bound explicit. Recall that the $j$-invariant of $E$ is $j_E=\frac{3^3c^3(9a^p+b^p)^3}{(ab^3)^p}$. Therefore $v_{\mathfrak P}(j_E)=3v_{\lambda}(3)-3pv_{\lambda}(b)<0$ and $p\nmid v_{\lambda}(j_E)$ when $p>5$. Therefore, by the theory of Tate curve, the inertia group $I_{\lambda}$ contains an element which acts on $E[p]$ via $\begin{pmatrix}
  	1& 1\\
  	0 & 1
  \end{pmatrix}$ which has order $p$. By Theorem \ref{subgroups}, the image $\overline{\rho}_{E,p}$  contains $\SL_2(\mathbb F_p)$ hence is an absolutely irreducible group of $\GL_2(\mathbb F_p)$. 

\item Assume $\lambda \nmid ab$. We will use the below theorem of Larson and Vaintrob, Theorem \ref{larson}

Note that by Theorem \ref{subgroups}, for big enough $p$ if $\overline{\rho}_{E,p}$ is irreducible but absolutely reducible then the image $\overline{\rho}_{E,p}(G_K)$ cannot be contained in Borel or an exceptional subgroup. Therefore it is contained in non-split Cartan or in a normalizer of Cartan subgroup but not in the Cartan itself. Say   $\overline{\rho}_{E,p}(G_K)$  has an element $g$ which is not in non-split Cartan but in the normalizer of non-split Cartan. Then any non-identity element in the nonsplit Cartan subgroup and $g$ do not share a common eigenvector, hence contradicts to the assumption that $G$ is absolutely irreducible. Similar argument can be made for split Cartan case.

Therefore we can conclude that $\overline{\rho}_{E,p}(G_K)$  is in a nonsplit Cartan subgroup. The rest of the argument is similar as in \cite{NT}. Up to conjugation,  $\overline{\rho}_{E,p} \otimes \mathbb{F}_p \sim \begin{pmatrix}
	\lambda & 0\\
	0 & \lambda^p
\end{pmatrix}$ where $\lambda : G_K \rightarrow \F_{p^2}$ and $\lambda^{p+1}=\chi_p$ where $\chi_p$ is the cyclotomic character.

	\begin{enumerate}
	\item Assume that  $p$ splits in $K$ and $p \equiv 3 \pmod 4$. We'll assume that $\bar{\rho}_{E,p}$ is irreducible but absolutely reducible and get a contradiction.
	We will apply the above theorem to the elliptic curve $E$ attached to a putative solution of the equation \eqref{maineqn}.  Assuming the prime $p$ is greater than the constant given in the theorem, we get an CM elliptic curve $E'/K$ with $\overline{\rho}_{E',p} \otimes \overline{\mathbb F}_p \sim \begin{pmatrix}
		\theta & 0\\
		0 & \theta'
	\end{pmatrix}$ such that $\theta^{12} = \lambda^{12}.$

Since $p$ splits in $K$, we see that the image $\overline{\rho}_{E',p}(G_K)$ is contained inside a split Cartan subgroup which implies that the character $\theta$ is in fact $\mathbb{F}_p$-valued. In particular the order of $\theta$ is divisible by $p-1.$

	We also know that by Theorem 1 in \cite{LV}, $\lambda \theta^{-1}$ is unramified away from the additive primes of $E$. Therefore $\lambda \theta^{-1}$ is unramified at $\mathcal P | p.$  Since $\lambda^{p+1}=\chi_p$ we get $\theta^{p+1}|I_{\mathcal P}=\chi_p|I_{\mathcal P}$. 	
	Note that $p \equiv  3 \pmod 4$, $\frac{p-1}{2}$ is odd. We deduce that $\chi_p^{(p-1)/2}|I_{\mathcal P} = (\theta^{p-1})^{(p+1)/2}| I_{\mathcal P} = 1$. However, the order of $\chi_p^{(p-1)/2}|I_{\mathcal P} $ is $p-1$ since it surjects on $\mathbb F_p^{*}$.

	\item If $\overline{\rho}_{E,p}(I_\lambda)$ is divisible by $3$ then there exists an element $g \in \overline{\rho}_{E,p}(I_\lambda)$ which has order $3$, hence $\lambda(g) \in \F_{p^2}^\times$ has order $3$. Also note that $\chi_p$ is unramified at $p$. Therefore $\chi_p(g)=\lambda^{p+1}(g)=1$ which implies that $3|p+1$. 
 \end{enumerate}
\item When $K$ is totally real then the absolute Galois group $G_K$ contains a complex conjugation. The image of this complex conjugation under $\bar{\rho}_{E,p}$ is similar to $\begin{pmatrix}
	1 & 0 \\
	0 & -1
\end{pmatrix}$ which imples that if $\bar{\rho}_{E,p}$ is irreducible then it is absolutely irreducible.
\end{enumerate}
 \end{proof}
  
The following theorem of Larson and Vaintrob is used in the proof of the above corollary.

\begin{theorem}[\cite{LV}, Theorem 1]\label{larson}
	Let $K$ be a number field. There exists a finite
	set of primes $M_K$, depending only on $K$, such that for any prime $p \notin M_K$ and any elliptic curve $E/K$ for which $\overline{\rho}_{E,p} \otimes \overline{\mathbb F}_p$ is conjugate to $\begin{pmatrix}
		\lambda & * \\
		0 & \lambda'
	\end{pmatrix}$ 
	where $\lambda, \lambda' :G \rightarrow \mathbb{F}_p^\times$ are
	characters, one of the following happens.
	\begin{enumerate}
		\item There exists an elliptic curve $E'/K$ with complex multiplication(CM) , whose CM field is contained in $K$, with $\overline{\rho}_{E',p} \otimes \overline{\mathbb F}_p$ is conjugate to $\begin{pmatrix}
			\theta & * \\
			0 & \theta'
		\end{pmatrix}$ and such that $\theta^{12}=\lambda^{12}.$
		\item The Generalized Riemann Hypothesis fails for  $K=\Q(\sqrt{-p})$ and $\theta^{12}=\chi_p^6$. Moreover, in this case $\overline{\rho}_{E',p}$ is already reducible over $\F_p$ and $p \equiv 3 \pmod 4.$ 
		
	\end{enumerate}
	
\end{theorem}

 \section{Proof of the Theorem \ref{mainthm1}} 
 In this section, we will prove Theorem \ref{mainthm1}.

 Let $K$ be a number field with $h^{+}_K=1$
 satisfying Conjectures \ref{conj1} and \ref{conj2} and containing $\Q(\zeta_{3})$ where $\zeta_{3}$ is a primitive $3^{rd}$ root of unity. Assume $\lambda$ is the only prime of $K$ above $3$. Let $B_K$ be as
 in Theorem \ref{thm:levelred}, and let $(a,b,c)\in W_K$ be a non-trivial solution to the Fermat equation with signature $(p,p,3)$ given in \eqref{maineqn}.  We now apply Theorem \ref{thm:levelred}
 and obtain an elliptic curve $E'/K$ having a $K$-rational point of order $3$ and (potentially) good reduction away from $\lambda$ with $j$-invariant
 $j'$ satisfying $v_{\lambda}(j')<0$.  However, by Theorem \ref{thm for tot ram} applied with $\ell=3$ there is no such an elliptic curve, which gives us a contradiction.   
 \begin{theorem}[\cite{FKS}, Theorem 1]\label{thm for tot ram}
 	Let $\ell$ be a rational prime. Let $K$ be a number field satisfying the following conditions:
 	\begin{itemize}
 		\item $\Q(\zeta_{\ell}) \subset K$, where $\zeta_{\ell}$ is a primitive $\ell^{th}$ root of unity;
 		\item $K$ has a unique prime $\lambda$ above $\ell$;
 		\item $gcd(h_K^+,\ell(\ell-1))=1$, where $h_K^+$ is the narrow class number of $K$.
 	\end{itemize}
 	Then there is no elliptic curve $E/K$ with a $K$-rational $\ell$-isogeny, good reduction away from $\lambda$, potentially multiplicative reduction at $\lambda$.
 \end{theorem}  
 \begin{remark}
Since we assume that $\Q(\zeta_{3}) \subset K$, the triple $(\zeta_3, \zeta^2_3,1)$ is a non-trivial solution to the equation (\ref{maineqn}) in $\mathcal O_K$. However, as $\zeta_3$ is a unit in $\mathcal{O}_K$, the prime $\lambda$ does not divide $(\zeta_3)=\mathcal{O}_K$, so $(\zeta_3, \zeta^2_3,1) \not\in W_K$. 
 \end{remark}

\subsection{Numerical Examples}  Let $\mathcal{F}_n$ denote the set of number fields of degree $n$ and class number $1$ with discriminant less than $D_n$, where $D_n$ is $10^8$ if $n=4,6,8$ and $10^{16}$ if $n=10,12$. Let $\mathcal{K}_n$ denote the subset of $\mathcal{F}_n$ such that if $K\in\mathcal{K}_n$ then $K$ satisfies the following
\begin{itemize}
\item $\Q(\zeta_3)\subset K$,
\item the narrow class number of $K$ is $1$,
\item there is only one prime ideal of $K$ lying above $3$.
\end{itemize}

We computed the complete sets $\mathcal{F}_n$ and $ \mathcal{K}_n$ for $n=4,6,8,10,12$ in the John Jones \href{https://hobbes.la.asu.edu/NFDB/}{\texttt{Number Field Database}} \cite{DB}. The results are summarised in Table \ref{tab:table5}, and can be found online at \url{https://sites.google.com/view/erman-isik/research?authuser=0}. 
\begin{table}[h!]
	\begin{center}
		\label{tab:table5}
		\begin{tabular}{|c|c|c|c|}
			\hline
			$n$ & $\mathcal{F}_n$ &  $\mathcal{K}_n$ \\
			\hline\hline
			$4$ & $998395$ &  $28750$ \\
			$6$ & $605497$ &  $20320$ \\
			$8$ & $26361$ &  $1264$ \\
			$10$ & $895218$ &  $51527$ \\
			$12$ & $67466$ &  $750$ \\
			\hline
		\end{tabular}
		
  \end{center}
\end{table}

It then follows from Theorem \ref{mainthm1} that, for any $K$ that belongs to $\mathcal{K}_n$, the asymptotic Fermat's Last Theorem holds for $W_K$.

 \section{Proof of the Theorem \ref{mainthm2}}\label{sectionformainthm2}
 
  In this section, we will prove Theorem~\ref{mainthm2}. One of the main steps towards the proof is to lift mod $p$ eigenforms to complex ones. Recall that it follows from Proposition~\ref{eigenform} that there is a constant $B(\frakN)$ such that for $p>B(\frakN)$, all mod $p$ eigenforms lift to complex ones. However, we want to make this bound explicit for the fields we consider for Theorem~\ref{mainthm2}. Before the proof of Theorem~\ref{mainthm2}, we will state the key point to overcome this difficulty. Note that our approach follows closely \cite[Section 2]{Turcas2018} and \cite[Section 2 and 3]{SS}.
  
  Let $K$ be a number field with the integer ring $\calO_{K}$, and let $\frakN$ be an ideal of $\calO_{K}$. Assume that $p$ is a rational prime unramified in $K$ and relatively prime to $\frakN$. Consider the following short exact sequence given by multiplication-by-$p$ map
 $$0 \rightarrow \mathbb{Z} \xrightarrow{\times p}\mathbb{Z} \rightarrow \mathbb{F}_p \rightarrow 0.$$
 
 This exact sequence gives rise to a long exact sequence on the cohomology groups from which we can extract the following short exact sequence
 \begin{equation}\label{exact_sequence}
 	0 \rightarrow H^1(Y_0(\frakN), \mathbb{Z})\otimes\mathbb{F}_p \xrightarrow{\;\delta\;}H^1(Y_0(\frakN), \mathbb{F}_p) \rightarrow H^2(Y_0(\frakN), \mathbb{Z})[p] \rightarrow 0,
 \end{equation}
 where $H^2(Y_0(\frakN), \mathbb{Z})[p]$ denotes the $p$-torsion subgroup of $H^2(Y_0(\frakN), \mathbb{Z})$. Hence, we deduce that $p$-torsion of $H^2(Y_0(\frakN),\mathbb{Z})$ vanishes if and only if the reduction map from $H^1(Y_0(\frakN),\mathbb{Z})$ to $H^1(Y_0(\frakN),\mathbb{F}_p)$ is surjective. As explained in \cite{SS} and \cite{Turcas2018}, we see that, for primes $p>3$, if the group $H^2(\Gamma_0(\frakN), \mathbb{Z})$ has a non-trivial $p$-torsion element, then  $\Gamma_0(\frakN)^{\rm ab}$ will have a $p$-torsion as well. If $H^2(Y_0(\frakN), \mathbb{Z})$ has only trivial $p$-torsion, then we deduce that the map
 $$H^1(Y_0(\frakN), \mathbb{Z})\otimes\mathbb{F}_p \xrightarrow{\;\delta\;}H^1(Y_0(\frakN), \mathbb{F}_p)$$
 is surjective. Therefore, any Hecke eigenvector in  $H^1(Y_0(\frakN), \mathbb{F}_p)$ comes from such an eigenvector in $H^1(Y_0(\frakN), \mathbb{Z})\otimes\mathbb{F}_p$. We can now utilize a lifting lemma of Ash and Stevens (\cite{AS86}, Proposition 1.2.2) to deduce that by fixing an embedding $\overline{\mathbb{Q}} \hookrightarrow \mathbb{C} $ we can regard a cohomology class in $H^1(Y_0(\frakN), \mathbb{Z})$ as a class in $H^1(Y_0(\frakN), \mathbb{C})$.

The existence of an eigenform (complex or mod $p$) is equivalent to the existence of a cohomology class in the corresponding cohomology group that is a simultaneous eigenvector for the Hecke operators such that its eigenvalues match the values of the eigenform. With this interpretation, we see that the mod $p$ eigenforms lift to complex eigenforms whenever the abelianization $\Gamma_0(\frakN)^{\rm ab}$ has only trivial $p$-torsion element.

 \subsubsection{\bfseries{Proof of Theorem \ref{mainthm2}}} Let $K=\mathbb{Q}(\sqrt{-d})$ with $d\in\{1,7,19,43,67\}$, and let $\lambda$ denote the prime ideal of $K$ lying above $3$. Suppose that $(a,b,c)\in\mathcal{O}_K^3$ is a non-trivial primitive solution to the equation (\ref{maineqn}). 
 
 Let  $\overline{\rho}_{E,p}$ be the residual Galois representation induced by the action of $G_K$ on $E[p]$. We want to apply Conjecture~\ref{conj1} to $\overline{\rho}_{E,p}$ . Note that in order to do this we need $\overline{\rho}_{E,p}$  to be absolutely irreducible. It follows from Corollary \ref{absirr} that  $\overline{\rho}_{E,p}$ is absolutely irreducible under the assumptions of Theorem \ref{mainthm2} hence satisfies the hypotheses of Conjecture~\ref{conj1}. This will be explained in Case I and Case II below. For now let's assume that Conjecture~\ref{conj1} is appliciable. Applying this conjecture,  we deduce that there exists a weight two, mod $p$
 eigenform $\theta$ over $K$ of level $\frakN_E$ such that for all primes $\frakq$ coprime to $p\frakN_E$, we have
 \[
 \tr(\overline{\rho}_{E,p}(\Frob_{\frakq}))=\theta(T_{\frakq}),
 \]
 where $T_{\frakq}$ denotes the Hecke operator at $\frakq$.
 
Recall that $\frakN_E$ denotes the Serre conductor of the residual representation $\overline{\rho}_{E,p}$, which is a power of $\lambda$. We now aim to lift this mod $p$ Bianchi modular form to a complex one.

 We compute the abelianizations $\Gamma_0(\frakN_E)^{\rm ab}$ implementing the algorithm of {\c{S}}eng{\"u}n \cite{Sen11}. One can access to the relevant \texttt{Magma} codes online at \url{https://warwick.ac.uk/fac/sci/maths/people/staff/turcas/fermatprog}. The results of the algorithm can be found at \url{https://sites.google.com/view/erman-isik/research?authuser=0}. We record here
 the primes $\ell$ that appear as orders of torsion elements of $\Gamma_0(\frakN_E)^{\rm ab}$ for each number field in Table~\ref{tab:primetorsion}.
 
 \begin{table}[h!]
 	\begin{center}
 		\label{tab:primetorsion}
 		\caption{prime torsions in $\Gamma_0(\frakN_E)^{\rm ab}$ }
 		\begin{tabular}{|l|c|r|}
 			\hline
 			Number Fields &  $val_\lambda(\frakN_E)$ & Primes $\ell$ such that $\Gamma_0(\frakN_E)^{\rm ab}[\ell] \neq 0$  \\
 			\hline\hline
 			$\mathbb{Q}(i)$ & 1,2,3 & 2,3 \\
 			$\mathbb{Q}(\sqrt{-7})$ &1,2,3 & 2,3  \\
 			$\mathbb{Q}(\sqrt{-19})$ &1,2,3& 2,3 ,5 \\
 			$\mathbb{Q}(\sqrt{-43})$ &1,2,3 & 2,3,5,59,67,199  \\
 			$\mathbb{Q}(\sqrt{-67})$ &1,2,3 &  2,3,5,17,19,37,47,67 \\
 			
 				\hline
 		\end{tabular}
 	\end{center}
 \end{table}

 Assume that $\ell_K$ is the largest prime in Table \ref{tab:primetorsion} related to the number field $K$, and that $p> \ell_K$.  It then follows that the $p$-torsion subgroups of $\Gamma_0(\frakN_E)^{\rm ab}$ are all trivial, so the mod $p$ eigenforms must lift to complex ones. The procedure explained at the beginning of Section~\ref{sectionformainthm2} together with Conjecture \ref{conj1} imply that there exists a (complex) Bianchi modular form $\frakf$ over $K$ of level $\frakN_E$ such that for all prime ideals $\mathfrak{q}$ coprime to $p\frakN_E$ we have
 
 $${\rm Tr}(\overline{\rho}_{E,p}({\rm Frob}_{\mathfrak{q}})) \equiv \mathfrak{f}(T_{\mathfrak{q}}) \pmod{\mathfrak{p}},$$
  where $\frakp$ is a prime ideal of $\mathbb{Q}_\frakf$ lying above $p$ and $\mathbb{Q}_\frakf$ is the number field generated by the eigenvalues. Let us denote this relation by $\overline{\rho}_{E,p}\sim \overline{\rho}_{\frakf,\frakp}$.
 
 Recall that the constants $C_K$ and $M_K$ were defined in Proposition \ref{irrCase2} and  in Corollary \ref{absirr}.

 \subsection{Case I:} Assume that the prime ideal $\lambda$ of $K$ lying above $3$ divides $b$. Note that in this case the associated Frey curve is semistable by Lemma \ref{semist}. By Proposition \ref{irrCase2} $\overline{\rho}_{E,p}$ is reducible when $p>C_K$. By Corollary \ref{absirr} Part (1), $\overline{\rho}_{E,p}$ is absolutely irreducible if it is irreducible. Therefore Conjecture ~\ref{conj1} is appliciable to $\overline{\rho}_{E,p}$. When we apply Conjecture~\ref{conj1} and the lifting argument above, we see that  the corresponding Bianchi modular form $\frakf$ is of level $1 $ or $\lambda$. Since there are no Bianchi newforms at these levels over $K$, it follows that Theorem \ref{mainthm2} holds true for $p>{\rm max} \{\ell_K, C_K\}$. This proves Case I of Theorem \ref{mainthm2}.

 \subsection{Case II:} In this case, we do not have the absolute irreducibility of $\overline{\rho}_{E,p}$ for all primes $p$ as illustrated in Corollary \ref{absirr}. Therefore we need the restrictions in the statement of Theorem \ref{mainthm2} Part II i.e. $p \equiv 3 \mod 4$ and $p$ splits in $K$. Under these restrictions and when $p> {\rm max}\{C_K, M_K\}$, $\overline{\rho}_{E,p}$ is absolutely irreducible by  Corollary \ref{absirr}. We also need $p> \ell_K $ to lift the mod $p$ eigenforms to complex ones as explained above. Therefore from now on we assume that $p> B_K={\rm max}\{C_K, M_K,\ell_K\}$. 
 
 Recall that in this case the associated Frey curve is semistable away from $\lambda$ and the power of $\lambda$ in the conductor of $E$ is $2$ or $3$ by Lemma \ref{semist}. Then the corresponding Bianchi modular form is of level dividing $\lambda^{3}$.

  \begin{lemma}\label{ideal_Bf}Let us fix a prime ideal $\mathfrak{q} \neq \lambda$ of $K$, and let $\frakf$ be a newform of level dividing $\lambda^{3}$. Define the following set 
 	$$\mathcal{A}(\frakq)=\{a \in \mathbb{Z}\; :\; |a|\leq 2\sqrt{\Norm(\mathfrak{q})},\; \Norm(\mathfrak{q})+1-a \equiv 0 \pmod 3\}.$$ 
 	If $\overline{\rho}_{E,p}\sim \overline{\rho}_{\frakf,\frakp}$, where $\frakp$ is the prime ideal of $\mathbb{Q}_{\frakf}$ lying above $p$, then $\frakp$ divides
 	$$B_{\frakf,\mathfrak{q}}:=\Norm(\mathfrak{q})\cdot\left(\Norm(\mathfrak{q} +1)^2-\mathfrak{f}(T_{\mathfrak{q}})^2\right)\cdot\prod_{a \in \mathcal{A}(\frakq)}(a-\mathfrak{f}(T_{\mathfrak{q}}))\mathcal{O}_{\mathbb{Q}_{\frakf}}.$$
 \end{lemma}
 
 \begin{proof}
 	If $\mathfrak{q}|p$, then $\Norm(\mathfrak{q})$ is a power of $p$. Now assume that $\mathfrak{q}$ does not divide $p$. Then the Frey curve $E$ has semistable reduction at $\mathfrak{q}$. If it has a good reduction, then we have
 	$${\rm Tr}(\overline{\rho}_{E,p}({\rm Frob}_{\mathfrak{q}})) \equiv a_{\mathfrak{q}}(E) \equiv \mathfrak{f}(T_{\mathfrak{q}}) \pmod{\frakp}.$$
 	
 	Note that by Lemma~\ref{semist} the Frey curve $E$ has a $3$-torsion point, so $3$ divides $|E(\mathbb{F}_{\mathfrak{q}})|=\Norm(\mathfrak{q})+1=a_{\mathfrak{q}}(E)$. By Hasse-Weil bound, we know that $|a_{\mathfrak{q}}(E)|\leq 2\sqrt{\Norm(\mathfrak{q})}$. So $a_{\mathfrak{q}}(E)$ belongs to the finite set $\mathcal{A}(\frakq)$. Finally, suppose that $E$ has multiplicative reduction at $\mathfrak{q}$. Then by comparing the traces of the images of Frobenius at $\mathfrak{q}$ under $\overline{\rho}_{E,p}$, we have
 	$$\pm(\Norm(\mathfrak{q})+1) \equiv \mathfrak{f}(T_{\mathfrak{q}}) \pmod \frakp.$$
 	It then follows that $\frakp$ divides $(\Norm(\mathfrak{q})+1)^2-\mathfrak{f}(T_{\mathfrak{q}})^2$. Hence, $\mathfrak{p}$ divides $B_{\frakf,\mathfrak{q}}$.
 \end{proof}
 
 Using \texttt{Magma}, we computed the cuspidal newforms at the predicted levels, the fields $\mathbb{Q}_\frakf$ and eigenvalues $\frakf(T_\frakq)$ at the prime ideals $\frakq$ of norm less that $50$ for each imaginary quadratic number field $K=\mathbb{Q}(\sqrt{-d})$ with $d\in\{1,7,19,43,67\}$.  For each modular form $\frakf$ of level dividing $\lambda^3$, we computed the ideal
 $$B_\frakf:=\sum_{\frakq\in S}B_{\frakf,\frakq}, $$
where $S$ denotes the set of prime ideals $\frakq\neq \lambda$ of $K$ of norm less than $50$. Set $C_\frakf:={\rm Norm}_{\mathbb{Q}_\frakf/\mathbb{Q}}(B_\frakf)$. It then follows from Lemma~\ref{ideal_Bf} that $p$ divides $C_\frakf$.

The algorithm that we implemented and the results for the following fields can be found online at \url{https://sites.google.com/view/erman-isik/research?authuser=0}.

 \begin{enumerate}[(i)]
 	\item If $K=\mathbb Q (i)$, then there is no Bianchi modular form at level $\lambda^\epsilon$ where $\epsilon=1,2$. There is only one  modular form at level $\lambda^3$ and this form has CM.
 	
 	\item If $K=\mathbb Q(\sqrt{-7})$, then there is no Bianchi modular form at level $\lambda^\epsilon$ with $\epsilon=1,2$. There are three modular forms at level $\lambda^3$. For one of these forms, $C_\frakf$ is divisible by $2$ and $5$ and for the other $C_{\frakf}$ is divisible by $2$ and $7$. For the third modular form we get $C_\frakf=0$, and it has CM. 
 	
 	\item If $K=\mathbb Q(\sqrt{-19})$, then there is no Bianchi modular form at level $\lambda$. There are two Bianchi modular forms at level $\lambda^2$. For both of these forms, $C_\frakf$ is divisible by $2$. There are $3$ modular forms at level $\lambda^3$. For one of these forms, $C_\frakf$ is divisible by $7$ and for the other $C_{\frakf}$ is divisible by $2$ and $17$. For the third modular form we get $C_\frakf=0$ and this modular form has CM.

 	\item If $K=\mathbb Q(\sqrt{-43})$, then there is no Bianchi modular form at level $\lambda$. There are two Bianchi modular forms at level $\lambda^2$. For one of these forms, $C_\frakf$ is divisible by $2$ and $5$ and for the other $C_\frakf$ is a power of $2$. There are $3$ modular forms at level $\lambda^3$. For one of these modular forms, $C_\frakf$'s is one and for the other modular form $C_\frakf$ is divisible by $2$ and $23$.  For the third modular form we get $C_\frakf=0$, which has CM.

 	\item If $K=\mathbb Q(\sqrt{-67})$, then there is no Bianchi modular form at level $\lambda$. There are two Bianchi modular forms at level $\lambda^2$. For two of these forms, $C_\frakf$ is divisible by $2$. At level $\lambda^3$ we find $3$ modular forms. For one of these modular forms, $C_\frakf$ is one and for the other $C_\frakf$ is divisible by $11$ and $19$. For the third modular form we get $C_\frakf=0$ and this modular form has CM.

 \end{enumerate}
 
 Since $p> B_K$, it follows from Lemma~\ref{ideal_Bf} that the Frey curve cannot correspond to Bianchi modular forms with $C_\frakf\neq 0$. If $C_\frakf =0$ then the associated elliptic curve has complex multiplication as discussed above. If the Frey curve $E$ has complex multiplication, the integrality of its $j$-invariant implies that $a$ and $b$ are units in $\mathcal{O}_K$ but it is easy to check that the equation (\ref{maineqn}) doesn’t have any solutions with $a,b \in \mathcal{O}_K^\times$.  Hence, the Frey curve cannot correspond to Bianchi modular forms with CM as well, which proves Theorem~\ref{mainthm2}.
 
\begin{remark}
    For simplicity, we only considered the fields where there is only one prime ideal lying above $3$, so we excluded the fields $\Q(\sqrt{-2})$ and $\Q(\sqrt{-11})$. Note that the triple $(\omega,\omega^2,1)$, where $\omega$ is a primitive $3^{\rm rd}$ root of unity, is a solution to the equation \eqref{maineqn}. Hence, we had to exclude the case $\Q(\sqrt{-3})$.
    
    One can try to apply the argument above to deal with the Fermat equation with signature $(p,p,3)$ defined over the imaginary quadratic field $\Q(\sqrt{-163})$. However, we were unable to compute all the Bianchi modular forms  over $\Q(\sqrt{-163})$ at level $\lambda^3$.
\end{remark}

 \section{Appendix-Ternary Equation of Signature $(p,p,2)$}
 
 In \cite{IKO} we have proved that the equation $x^p+y^p=z^2$ has no solutions asymptotically where $2|b$ over certain number fields. Using the method in the proof of Theorem \ref{mainthm1}, we can extend the results to any number field $K$ with $h_K^+=1$. The difference relies in the proof method. Generally speaking to solve a Diophantine equation using modular method one needs to do either  compute newforms of a certain level or  compute all elliptic curves of a given conductor and particular information about torsion subgroup and/or rational isogeny or 
compute all solutions to an $S$-unit equation.

In \cite{IKO} we used the $S$-unit equation  method and this restricted us to the totally real number fields. However using Theorem \ref{thm for tot ram} as we did for proving Theorem \ref{mainthm1} we can get the following result about the solutions of $x^p+y^p=z^2$.

\begin{theorem}\label{thm:pp2}
	Let $K$ be a number field with narrow class number $h_K^{+}=1$ satisfying the conjectures \ref{conj1} and \ref{conj2}.  Assume $\beta$ is the only prime of $K$ lying above $2$. Let $W_K$ be the set of $(a,b,c)\in\calO_K$
	such that $a^p+b^p=c^2$ with $\beta|b$. Then there is a constant $B_K$ -depending only on $K$- such that for $p > B_K$, equation \eqref{maineqn} has no solution $(a,b,c) \in W_K$. 
	\end{theorem}

A hypothetical solution $(a,b,c) \in W_K$ with exponent $p>B_K$ gives rise to an elliptic curve $E'/K$ with a $K$-rational $2$-isogeny, good reduction away from $\frakP$ and potentially multiplicative reduction at $\frakP$. But this contradicts with Theorem \ref{thm for tot ram} above, hence there is no such a solution. 

Now we give some examples of totally real number fields with $h_K^{+}=1$, in which $2$ is totally ramified.

\textbf{Degree $2^n$ case.}  Let $f_1(x)=x^2-2$, and $f_n(x)= (f_{n-1}(x))^2-2$ for $n \geq 2$. Let $K_n=\Q(\theta_n)$, where $\theta_n $ is a root of the polynomial $f_n(x)$. Then it can be seen that $K_n=\Q(\sqrt{2+ \ldots +\sqrt{2}}) = \Q(\zeta_{2^{n+2}}+\zeta_{2^{n+2}}^{-1})$, the maximal totally real subfield of $K_n$. Then, $K_n$ is a totally real number field of degree $2^n$ in which $2$ is totally ramified. For $2 \leq n \leq 5$ it is possible to check that the narrow class number of $K_n$ is $1$, hence it follows from Corollary 6.5 in \cite{IKO} that the asymptotic FLT holds for all $K_n$ with $ 2 \leq n \leq 5$. For some other higher values of $n$ it is only conjecturally true that  the narrow class number is $1$ (using \texttt{MAGMA} under GRH).

\textbf{Degree $3,4,5$ case.} For $n \geq 3$, let $\mathcal{F}_n$ be the set of totally real number fields of degree $n$, discriminant $\leq 10^6$, in which $2$ totally ramifies. We are able to find the complete sets $\mathcal{F}_3, \mathcal{F}_4$ and $\mathcal{F}_5$ in the John Jones \href{https://hobbes.la.asu.edu/NFDB/}{\texttt{Number Field Database}} \cite{DB}. We define the following sets:
\begin{itemize}
	\item Let $\mathcal{G}_n$ be the set of $K \in \mathcal{F}_n$ such that $h_K^+$ is odd.
	\item Let $\mathcal{K}_n$ be the set of $K \in \mathcal{G}_n$ such that $h_K^+=1$.
\end{itemize}

Of course $ \mathcal{K}_n \subseteq \mathcal{G}_n \subseteq \mathcal{F}_n$, and the asymptotic FLT holds for any $K$ belonging to $\mathcal{K}_n$ by Corollary 6.5 in \cite{IKO}. The results are summarised in the following table, and can be found online at \url{https://sites.google.com/view/erman-isik/research?authuser=0}. 
\begin{table}[h!]
	\begin{center}
		\label{tab:table3}
		
		\begin{tabular}{|c|c|c|c|}
			\hline\hline 
			
			$n$ & $\mathcal{F}_n$ & $\mathcal{G}_n$ & $\mathcal{K}_n$ \\
			\hline\hline
			$3$ & $8600$ & $3488$ & $3046$ \\
			$4$ & $1243$ & $1$ & $1$ \\
			$5$ & $23$ & $13$ & $13$ \\
			\hline\hline
		\end{tabular}

	\end{center}
\end{table}

 \bibliographystyle{plain}
 \bibliography{ref.bib}

\begin{bibdiv}
\begin{biblist}

\bib{AS86}{article}{
      author={Ash, Avner},
      author={Stevens, Glenn},
       title={Cohomology of arithmetic groups and congruences between systems
  of {H}ecke eigenvalues},
        date={1986},
        ISSN={0075-4102},
     journal={J. Reine Angew. Math.},
      volume={365},
       pages={192\ndash 220},
      review={\MR{826158}},
}

\bib{BVY}{article}{
      author={Bennett, Michael~A.},
      author={Vatsal, Vinayak},
      author={Yazdani, Soroosh},
       title={Ternary {D}iophantine equations of signature {$(p,p,3)$}},
        date={2004},
        ISSN={0010-437X},
     journal={Compos. Math.},
      volume={140},
      number={6},
       pages={1399\ndash 1416},
  url={https://doi-org.ezproxy.lib.utexas.edu/10.1112/S0010437X04000983},
      review={\MR{2098394}},
}

\bib{SS}{article}{
      author={\c{S}eng\"{u}n, Mehmet~Haluk},
      author={Siksek, Samir},
       title={On the asymptotic {F}ermat's last theorem over number fields},
        date={2018},
        ISSN={0010-2571},
     journal={Comment. Math. Helv.},
      volume={93},
      number={2},
       pages={359\ndash 375},
         url={https://doi-org.ezproxy.lib.utexas.edu/10.4171/CMH/437},
      review={\MR{3811755}},
}

\bib{ECbook}{proceedings}{
      editor={Coates, John},
      editor={Yau, S.~T.},
       title={Elliptic curves, modular forms \& {F}ermat's last theorem},
     edition={Second},
   publisher={International Press, Cambridge, MA},
        date={1997},
        ISBN={1-57146-049-7},
      review={\MR{1605709}},
}

\bib{David2012}{misc}{
      author={David, Agnès},
       title={Caract\`ere d'isog\'enie et crit\`eres d'irr\'eductibilit\'e},
        date={2012},
}

\bib{HD}{article}{
      author={Deconinck, Heline},
       title={On the generalized {F}ermat equation over totally real fields},
        date={2016},
        ISSN={0065-1036},
     journal={Acta Arith.},
      volume={173},
      number={3},
       pages={225\ndash 237},
         url={https://doi-org.ezproxy.lib.utexas.edu/10.4064/aa8171-1-2016},
      review={\MR{3512853}},
}

\bib{DG}{article}{
      author={Darmon, Henri},
      author={Granville, Andrew},
       title={On the equations {$z^m=F(x,y)$} and {$Ax^p+By^q=Cz^r$}},
        date={1995},
        ISSN={0024-6093},
     journal={Bull. London Math. Soc.},
      volume={27},
      number={6},
       pages={513\ndash 543},
         url={https://doi-org.ezproxy.lib.utexas.edu/10.1112/blms/27.6.513},
      review={\MR{1348707}},
}

\bib{der2021torsion}{misc}{
      author={Derickx, Maarten},
      author={Kamienny, Sheldon},
      author={Stein, William},
      author={Stoll, Michael},
       title={Torsion points on elliptic curves over number fields of small
  degree},
        date={2021},
}

\bib{localglobal}{misc}{
      author={Etropolski, Anastassia},
       title={Local-global principles for certain images of galois
  representations},
        date={2015},
}

\bib{FKS}{article}{
      author={Freitas, Nuno},
      author={Kraus, Alain},
      author={Siksek, Samir},
       title={Class field theory, {D}iophantine analysis and the asymptotic
  {F}ermat's last theorem},
        date={2020},
        ISSN={0001-8708},
     journal={Adv. Math.},
      volume={363},
       pages={106964, 37},
  url={https://doi-org.ezproxy.lib.utexas.edu/10.1016/j.aim.2019.106964},
      review={\MR{4054049}},
}

\bib{FS}{article}{
      author={Freitas, Nuno},
      author={Siksek, Samir},
       title={The asymptotic {F}ermat's last theorem for five-sixths of real
  quadratic fields},
        date={2015},
        ISSN={0010-437X},
     journal={Compos. Math.},
      volume={151},
      number={8},
       pages={1395\ndash 1415},
  url={https://doi-org.ezproxy.lib.utexas.edu/10.1112/S0010437X14007957},
      review={\MR{3383161}},
}

\bib{FSANT}{article}{
      author={Freitas, Nuno},
      author={Siksek, Samir},
       title={Fermat's last theorem over some small real quadratic fields},
        date={2015},
        ISSN={1937-0652},
     journal={Algebra Number Theory},
      volume={9},
      number={4},
       pages={875\ndash 895},
         url={https://doi-org.ezproxy.lib.utexas.edu/10.2140/ant.2015.9.875},
      review={\MR{3352822}},
}

\bib{IKO}{article}{
      author={I\c{s}ik, Erman},
      author={Kara, Yasemin},
      author={Ozman, Ekin},
       title={On ternary {D}iophantine equations of signature {$(p, p, 2)$}
  over number fields},
        date={2020},
        ISSN={1300-0098},
     journal={Turkish J. Math.},
      volume={44},
      number={4},
       pages={1197\ndash 1211},
         url={https://doi-org.ezproxy.lib.utexas.edu/10.3906/mat-1911-88},
      review={\MR{4122899}},
}

\bib{DB}{article}{
      author={Jones, J.\~W.},
      author={Roberts, D.\~P.},
       title={A database of number fields},
        date={2014},
      volume={17},
       pages={595–618},
}

\bib{Kam}{article}{
      author={Kamienny, S.},
       title={Torsion points on elliptic curves and {$q$}-coefficients of
  modular forms},
        date={1992},
        ISSN={0020-9910},
     journal={Invent. Math.},
      volume={109},
      number={2},
       pages={221\ndash 229},
         url={https://doi-org.ezproxy.lib.utexas.edu/10.1007/BF01232025},
      review={\MR{1172689}},
}

\bib{Katz}{article}{
      author={Katz, Nicholas~M.},
       title={Galois properties of torsion points on abelian varieties},
        date={1981},
        ISSN={0020-9910},
     journal={Invent. Math.},
      volume={62},
      number={3},
       pages={481\ndash 502},
         url={https://doi-org.ezproxy.lib.utexas.edu/10.1007/BF01394256},
      review={\MR{604840}},
}

\bib{KM}{article}{
      author={Kenku, M.~A.},
      author={Momose, F.},
       title={Torsion points on elliptic curves defined over quadratic fields},
        date={1988},
        ISSN={0027-7630},
     journal={Nagoya Math. J.},
      volume={109},
       pages={125\ndash 149},
  url={https://doi-org.ezproxy.lib.utexas.edu/10.1017/S0027763000002816},
      review={\MR{931956}},
}

\bib{KO}{article}{
      author={Kara, Yasemin},
      author={Ozman, Ekin},
       title={Asymptotic generalized {F}ermat's last theorem over number
  fields},
        date={2020},
        ISSN={1793-0421},
     journal={Int. J. Number Theory},
      volume={16},
      number={5},
       pages={907\ndash 924},
  url={https://doi-org.ezproxy.lib.utexas.edu/10.1142/S1793042120500463},
      review={\MR{4101578}},
}

\bib{Kraus90}{article}{
      author={Kraus, Alain},
       title={Sur le d\'{e}faut de semi-stabilit\'{e} des courbes elliptiques
  \`a r\'{e}duction additive},
        date={1990},
        ISSN={0025-2611},
     journal={Manuscripta Math.},
      volume={69},
      number={4},
       pages={353\ndash 385},
         url={https://doi-org.ezproxy.lib.utexas.edu/10.1007/BF02567933},
      review={\MR{1080288}},
}

\bib{LV}{article}{
      author={Larson, Eric},
      author={Vaintrob, Dmitry},
       title={Determinants of subquotients of {G}alois representations
  associated with abelian varieties},
        date={2014},
        ISSN={1474-7480},
     journal={J. Inst. Math. Jussieu},
      volume={13},
      number={3},
       pages={517\ndash 559},
  url={https://doi-org.ezproxy.lib.utexas.edu/10.1017/S1474748013000182},
        note={With an appendix by Brian Conrad},
      review={\MR{3211798}},
}

\bib{Mocanu}{misc}{
      author={Mocanu, Diana},
       title={Asymptotic fermat for signatures $(p, p, 2)$ and $(p, p, 3)$ over
  totally real fields},
  how={\url{https://warwick.ac.uk/fac/sci/maths/people/staff/mocanu/summer_project.pdf}},
}

\bib{NT}{article}{
      author={Najman, Filip},
      author={\c{T}urca\c{s}, George~C.},
       title={Irreducibility of mod {$p$} {G}alois representations of elliptic
  curves with multiplicative reduction over number fields},
        date={2021},
        ISSN={1793-0421},
     journal={Int. J. Number Theory},
      volume={17},
      number={8},
       pages={1729\ndash 1738},
  url={https://doi-org.ezproxy.lib.utexas.edu/10.1142/S1793042121500585},
      review={\MR{4310203}},
}

\bib{pap}{article}{
      author={Papadopoulos, Ioannis},
       title={Sur la classification de {N}\'{e}ron des courbes elliptiques en
  caract\'{e}ristique r\'{e}siduelle {$2$} et {$3$}},
        date={1993},
        ISSN={0022-314X},
     journal={J. Number Theory},
      volume={44},
      number={2},
       pages={119\ndash 152},
         url={https://doi-org.ezproxy.lib.utexas.edu/10.1006/jnth.1993.1040},
      review={\MR{1225948}},
}

\bib{Sen11}{article}{
      author={\c{S}eng\"{u}n, Mehmet~Haluk},
       title={On the integral cohomology of {B}ianchi groups},
        date={2011},
        ISSN={1058-6458},
     journal={Exp. Math.},
      volume={20},
      number={4},
       pages={487\ndash 505},
  url={https://doi-org.ezproxy.lib.utexas.edu/10.1080/10586458.2011.594671},
      review={\MR{2859903}},
}

\bib{SD}{inproceedings}{
      author={Swinnerton-Dyer, H. P.~F.},
       title={On {$l$}-adic representations and congruences for coefficients of
  modular forms},
        date={1973},
   booktitle={Modular functions of one variable, {III} ({P}roc. {I}nternat.
  {S}ummer {S}chool, {U}niv. {A}ntwerp, 1972)},
       pages={1\ndash 55. Lecture Notes in Math., Vol. 350},
      review={\MR{0406931}},
}

\bib{Serre}{article}{
      author={Serre, Jean-Pierre},
       title={Divisibilit\'e de certaines fonctions arithm\'etiques},
        date={1976},
     journal={Enseignement Math. (2)},
      volume={22},
      number={3-4},
       pages={227\ndash 260},
      review={\MR{0434996 (55 \#7958)}},
}

\bib{Sil94}{book}{
      author={Silverman, Joseph~H.},
       title={Advanced topics in the arithmetic of elliptic curves},
      series={Graduate Texts in Mathematics},
   publisher={Springer-Verlag, New York},
        date={1994},
      volume={151},
        ISBN={0-387-94328-5},
  url={https://doi-org.ezproxy.lib.utexas.edu/10.1007/978-1-4612-0851-8},
      review={\MR{1312368}},
}

\bib{Turcas2018}{article}{
      author={\c{T}urca\c{s}, George~C.},
       title={On {F}ermat's equation over some quadratic imaginary number
  fields},
        date={2018},
        ISSN={2522-0160},
     journal={Res. Number Theory},
      volume={4},
      number={2},
       pages={Paper No. 24, 16},
  url={https://doi-org.ezproxy.lib.utexas.edu/10.1007/s40993-018-0117-y},
      review={\MR{3798168}},
}

\bib{Turcas2020}{article}{
      author={\c{T}urca\c{s}, George~C.},
       title={On {S}erre's modularity conjecture and {F}ermat's equation over
  quadratic imaginary fields of class number one},
        date={2020},
        ISSN={0022-314X},
     journal={J. Number Theory},
      volume={209},
       pages={516\ndash 530},
  url={https://doi-org.ezproxy.lib.utexas.edu/10.1016/j.jnt.2019.08.011},
      review={\MR{4053081}},
}

\end{biblist}
\end{bibdiv}
 
 \end{document}